\documentclass[bibalpha]{amsart}
\usepackage{amssymb}
\usepackage{textcomp}
\usepackage[mathscr]{euscript}
\usepackage{pb-diagram}
\usepackage{enumitem}
\usepackage{comment}

\newcommand{\dprk}{\operatorname{dp-rk}}

\newtheorem{theorem}{Theorem}[section]
\newtheorem{qst}[theorem]{Question}
\newtheorem{thm}[theorem]{Theorem}
\newtheorem{prop}[theorem]{Proposition}
\newtheorem{lem}[theorem]{Lemma}

\newtheorem{fact}[theorem]{Fact}

\newtheorem{lemma}[theorem]{Lemma}

\newtheorem{conj}{Conjecture}

\theoremstyle{definition}

\theoremstyle{remark}

\newcommand{\mfrak}{\mathfrak{m}}
\newcommand{\valp}{\mathrm{Val}_p}
\newcommand{\pring}{(\Z_p,+,\times}

\newcommand{\rad}{\operatorname{rad}}

\ProvideTextCommandDefault{\cprime}{(U+042C)}

\newcommand{\ksh}{K^{\mathrm{Sh}}}

\newcommand{\sub}{\subseteq}

\newcommand{\st}{\operatorname{st}}

\newcommand{\Sh}[1]{\ensuremath{\mathscr{#1}^{\mathrm{Sh}}}}

\newcommand{\nip}{\mathrm{NIP}}

\newcommand{\Cal}[1]{\ensuremath{\mathcal{#1}}}
\newcommand{\Sa}[1]{\ensuremath{\mathscr{#1}}}

\newcommand{\B}{\mathbb{B}}
\newcommand{\Z}{\mathbb{Z}}
\newcommand{\N}{\mathbb{N}}

\newcommand{\Q}{\mathbb{Q}}
\newcommand{\R}{\mathbb{R}}
\newcommand{\F}{\mathbb{F}}

\DeclareMathOperator{\Gl}{Gl}

\begin{document}
\title[]{A P-adic structure which does not interpret an infinite field but whose Shelah completion does}

\author{Erik Walsberg}
\address{Department of Mathematics, Statistics, and Computer Science\\
Department of Mathematics\\University of California, Irvine, 340 Rowland Hall (Bldg.\# 400),
Irvine, CA 92697-3875}
\email{ewalsber@uci.edu}
\urladdr{http://www.math.illinois.edu/\textasciitilde erikw}

\date{\today}

\maketitle

\begin{abstract}
We give a $p$-adic example of a structure whose Shelah completion interprets $\mathbb{Q}_p$ but which does not (provided an extremely plausible conjecture holds) interpret an infinite field.
In the final section we discuss the significance of such examples for a possible future geometric theory of $\mathrm{NIP}$ structures.
\end{abstract}

\section{Introduction}
\noindent
In \cite{big-nip,field-in-Shelah} we described $\nip$ structures $\Sa H$ such that $\Sa H$ does not interpret an infinite field but the Shelah completion of $\Sa H$ does.
Here we describe a more natural example, modulo a reasonable conjecture.
Fix a prime $p$ and let $K$ be a $(2^{\aleph_0})^+$-saturated elementary extension of $\Q_p$.
Let $\valp : K^\times \to \Gamma$ be the $p$-adic valuation on $K$ and $V$ be the valuation ring of $\valp$, recall that $\valp$ is $K$-definable.
Given $a \in K$ and $t \in \Gamma$ let $B(a,t)$ be the ball with center $a$ and radius $t$, i.e. the set of $a' \in K$ such that $\valp(a - a') \geq t$.
Let $\B$ be the set of balls in $V$.
Let $\Gamma_\geq$ be the set of nonnegative elements of $\Gamma$.
Let $\sim$ be the equivalence relation on $V \times \Gamma_\geq$ where $(a,t) \sim (a',t')$ if and only if $B(a,t) = B(a',t')$. 
Identify $\B$ with $(V \times \Gamma_\geq)/\!\sim$, consider $\B$ to be a $K$-definable set of imaginaries, and let $\Sa B$ be the structure induced on $\B$ by $K$.

\begin{thm}
\label{thm:2}
The Shelah completion of $\Sa B$ interprets $\Q_p$.
\end{thm}

\noindent
Conjecture~\ref{conj:obv} is a well-known and well-believed folklore conjecture.
Conjecture~\ref{conj:obv} is beyond the reach of current techniques, but its failure would be a huge surprise.
The analogue of Conjecture~\ref{conj:obv} for $\mathrm{ACVF}$ is a result of Hrushovski and Rideau~\cite{metastable}.

\begin{conj}
\label{conj:obv}
Any infinite field interpretable in $K$ is $K$-definably isomorphic to some finite extension of $K$.
\end{conj}

\noindent
The proof of Theorem~\ref{thm:1} is easy.

\begin{thm}
\label{thm:1}
If Conjecture~\ref{conj:obv} holds then $\Sa B$ does not interpret an infinite field.
\end{thm}

\noindent
Suppose that $\Sa O$ is a structure, $A$ is a subset of $M^m$, and $\tau : O \to A$ is a bijection.
We say that $\Sa M$ trace defines $\Sa O$ via $\tau$ if for every $\Sa O$-definable $X \sub O^n$ there is an $\Sa M$-definable $Y \subseteq M^{mn}$ such that 
$$ X = \{ (a_1,\ldots,a_n) \in O^n : (\tau(a_1),\ldots,\tau(a_n)) \in Y \}, $$
and $\Sa M$ \textbf{trace defines} $\Sa O$ if $\Sa M$ trace defines $\Sa O$ via some injection $\tau : O \to M^m$.

\begin{thm}
\label{thm:trace-field}
$\Sa B$ trace defines $\Q_p$.
\end{thm}

\noindent
Theorem~\ref{thm:trace-field} follows from Theorem~\ref{thm:2} and Theorem~\ref{thm:field-ext}.

\begin{thm}
\label{thm:field-ext}
Let $\lambda$ be a cardinal.
Suppose $\Sa M$ is $\lambda$-saturated and $\nip$.
If $\Sa O$ is interpretable in the Shelah completion of $\Sa M$ and $|O| < \lambda$ then $\Sa M$ trace defines $\Sa O$.
\end{thm}

\section{$\nip$-theoretic background}
\subsection{Shelah completeness}
Let $\Sa M$ be a structure and $\Sa M \prec \Sa N$ be $|M|^+$-saturated.
A subset $X$ of $M^n$ is \textbf{externally definable} if $X = M^n \cap Y$ for some $\Sa N$-definable subset $Y$ of $N^n$.
An application of saturation shows that the collection of externally definable sets does not depend on choice of $\Sa N$.
Fact~\ref{fact:convex} is well-known and easy.

\begin{fact}
\label{fact:convex}
Suppose that $X$ is an $\Sa M$-definable set and $<$ is an $\Sa M$-definable linear order on $X$.
Then any $<$-convex subset of $X$ is externally definable.
\end{fact}

\noindent
Fact~\ref{fact:cs} is arguably the most important result on externally definable sets at present.
It is a theorem of Chernikov and Simon~\cite{CS-II}.
The right to left implication is a saturation exercise which does not require $\nip$.

\begin{fact}
\label{fact:cs}
Suppose that $\Sa M$ is $\nip$ and $X$ is a subset of $M^n$.
Then $X$ is externally definable if and only if there is an $\Sa M$-definable family $(X_a : a \in M^m)$ of subsets of $M^n$ such that for every finite $A \subseteq X$ we have $A \subseteq X_a \subseteq X$ for some $a \in M^m$.
\end{fact}

\noindent
We say that a structure is \textbf{Shelah complete} if every externally definable set is already definable.
Note that $\Sa M$ is Shelah complete if and only if all types over $\Sa M$ are definable.
It follows that a theory $T$ is stable if and only if every model of $T$ is Shelah complete.
The Marker-Steinhorn theorem~\cite{Marker-Steinhorn} shows that if $\Sa O$ is an o-minimal expansion of a dense linear order $(O,<)$ then $\Sa O$ is Shelah complete if and only if $(O,<)$ is a complete linear order.
Results from dp-minimality show that $(\Z,+,<)$ is Shelah complete, see~\cite{SW-dp}.
Fact~\ref{fact:delon} is a theorem of Delon~\cite{Delon-def}.

\begin{fact}
\label{fact:delon}
$\Q_p$ is Shelah complete.
\end{fact}

\noindent
The \textbf{Shelah completion} $\Sh M$ of $\Sa M$ is the expansion of $\Sa M$ by all externally definable sets.
Fact~\ref{fact:shelah} is a theorem of Shelah~\cite{Shelah-external} and also a corollary to Fact~\ref{fact:cs}.

\begin{fact}
\label{fact:shelah}
If $\Sa M$ is $\nip$ then every $\Sh M$-definable set is externally definable in $\Sa M$.
\end{fact}

\noindent
Fact~\ref{fact:completion-is-complete} is an easy corollary to Fact~\ref{fact:shelah}, we leave the proof to the reader.

\begin{fact}
\label{fact:completion-is-complete}
The Shelah completion of a $\nip$ structure is Shelah complete.
\end{fact}

\noindent
The Shelah completion is usually called the ``Shelah expansion".
We believe that Fact~\ref{fact:completion-is-complete} justifies our nonstandard terminology.

\subsection{Dp-rank}
We will use a few basic facts about dp-rank.
See \cite{Simon-Book} for an overview of the dp-rank.
We let $\dprk_{\Sa M} X$ be the dp-rank of an $\Sa M$-definable set $X$.
The dp-rank of $\Sa M$ is defined to be $\dprk_{\Sa M} M$.
The first three claims of Fact~\ref{fact:dp-rank} are immediate consequences of the definition of dp-rank.
The fourth is essentially due to Onshuus and Usvyatsov~\cite{OnUs}.

\begin{fact}
\label{fact:dp-rank}
Suppose $X,Y$ are $\Sa M$-definable sets.
Then
\begin{enumerate}
\item $\dprk \Sa M < \infty$ if and only if $\Sa M$ is $\nip$,
\item $\dprk_{\Sa M} X = 0$ if and only if $X$ is finite,
\item If $f : X \to Y$ is a $\Sa M$-definable surjection then $\dprk_{\Sa M} Y \leq \dprk_{\Sa M} X$,
\item $\dprk_{\Sh M} X = \dprk_{\Sa M} X$ when $\Sa M$ is $\nip$.
\end{enumerate}
\end{fact}

\noindent
Fact~\ref{fact:dp-qp} is a theorem of Dolich, Goodrick, and Lippel~\cite{dp-basic}.

\begin{fact}
\label{fact:dp-qp}
The dp-rank of $\Q_p$ is one.
\end{fact}

\section{$\Sa B$}
\subsection{Proof of Theorem~\ref{thm:2}}
We identify the minimal positive element of $\Gamma$ with $1$ so that $\Z$ is the minimal non-trivial convex subgroup of $\Gamma$.
By Fact~\ref{fact:convex} $\Z$ is a $\ksh$-definable subset of $\Gamma$.
Let $v : K^\times \to \Gamma/\Z$ be the composition of $\valp$ with the quotient $\Gamma \to \Gamma/\Z$.
We equip $\Gamma/\Z$ with a group order by declaring $a + \Z \leq b + \Z$ when $a \leq b$, so $v$ is a $\ksh$-definable valuation on $K$.
Let $W$ be the valuation ring of $v$ and $\mfrak_W$ be the  maximal ideal of $W$.
So $W$ is the set of $a \in K$ such that $\valp(a) \geq m$ for some $m \in \Z$ and $\mfrak_W$ is the set of $a \in K$ such that $\valp(a) > m$ for all $m \in \Z$.
It is easy to see that for every $a \in W$ there is a unique $a' \in \Q_p$ such that $a - a' \in \mfrak_W$.
We identify $W/\mfrak_W$ with $\Q_p$ so that the residue map $W \to \Q_p$ is the usual standard part map $\st$.
So $\Q_p$ is a $\ksh$-definable set of imaginaries.
\newline

\noindent
Lemma~\ref{lem:st} is routine and left to the reader.

\begin{lem}
\label{lem:st}
Suppose that $X$ is a closed $\Q_p$-definable subset of $\Z^n_p$ and $X'$ is the subset of $V^n$ defined by any formula defining $X$.
Then $\st(X')$ agrees with $X$.
\end{lem}

\noindent
Given $B \in \B$ such that $B = B(a,t)$ we let $\rad(B) = t$.
As $\rad : \B \to \Gamma_\geq$ is surjective and $K$-definable we consider $\Gamma_\geq$ to be an imaginary sort of $\Sa B$ and $\rad$ to be a $\Sa B$-definable function.
Fix $\gamma \in \Gamma$ such that $\gamma > \N$.
Let $E$ be the set of $B \in \B$ such that $\rad(B) = \gamma$, so $E$ is $\Sa B$-definable.
\newline

\begin{proof}
It suffices to show that $\Sh B$ interprets $\pring)$.
Note that for any $a \in V$ and $b \in B(a,\gamma)$ we have $\st(a) = \st(b)$.
So we define a surjection $f : E \to \Z_p$ by declaring $f( B(a,\gamma) ) := \st(a)$ for all $a \in V$.
Let $\approx$ be the equivalence relation on $E$ where $B_0 \approx B_1$ if and only if $f(B_0) = f(B_1)$.
Note that for any $B_0,B_1 \in \B$ we have $B_0 \approx B_1$ if and only if 
$$ \{ B' \in \B : \rad(B') \in \N , B_0 \subseteq B' \} = \{ B' \in \B : \rad(B') \in \N , B_1 \subseteq B' \}. $$
So $\approx$ is $\Sh B$-definable.
Let $f : E^n \to \Z^n_p$ be given by 
$$ f(B_1,\ldots,B_n) = (f(B_1),\ldots,f(B_n)) \quad \text{for all} \quad B_1,\ldots,B_n \in E. $$
Suppose that $X$ is a $\Q_p$-definable subset of $\Z^n_p$.
We show that $f^{-1}(X)$ is $\Sh B$-definable.
As $X$ is $\Q_p$-definable it is a boolean combination of closed $\Q_p$-definable subsets of $\Z^n_p$, so we may suppose that $X$ is closed.
Let $X'$ be the subset of $V^n$ defined by any formula defining $X$.
Let $Y_0$ be the set of $B \in E$ such that $B \cap X' \neq \emptyset$ and $Y$ be the set of $B \in E$ such that $B \approx B_0$ for some $B_0 \in Y_0$.
Observe that $Y_0$ is $\Sa B$-definable and $Y$ is $\Sh B$-definable.
Note that $Y_0$ is the set of balls of the form $B(a,\gamma)$ for $a \in X'$.
\newline

\noindent
We show that $Y = f^{-1}(X)$.
Suppose $B(a,\gamma) \in f^{-1}(X)$.
So $\st(a) \in X$.
We have $B(\st(a),\gamma) \in Y_0$ and $B(\st(a),\gamma) \approx B(a,\gamma)$, so $B(a,\gamma) \in Y$.
Now suppose that $B(a,\gamma) \in Y$ and fix $B(b,\gamma) \in Y_0$ such that $B \approx B_0$.
We may suppose that $b \in X'$.
As $X$ is closed an application of Lemma~\ref{lem:st} shows that $\st(b) \in X$.
So $\st(a) = \st(b) \in X$.
So $B(a,\gamma) \in f^{-1}(X)$.
\end{proof}

\newpage

\subsection{Proof of Theorem~\ref{thm:1}}
We first prove an easy lemma.

\begin{lem}
\label{lem:map}
Any $K$-definable function $\B \to K^m$ has finite image.
\end{lem}

\begin{proof}
If $f : \B \to K^m$ has infinite image then there is a coordinate projection $e : K^m \to K$ such that $e \circ f$ has infinite image.
So we suppose $m = 1$.
Recall that if $X$ is a $K$-definable subset of $V$ then $X$ is either finite or has interior.
So it suffices to show that the image of any $K$-definable function $\B \to K$ has empty interior.
Let $\B(\Q_p)$ be the set of $\valp$-balls in $\Z_p$.
It is enough to show that the image of any function $\B(\Q_p) \to \Q_p$ has empty interior.
This holds as $\B(\Q_p)$ is countable and every nonempty open subset of $\Z_p$ is uncountable.
\end{proof}

\noindent
We now prove Theorem~\ref{thm:1}.

\begin{proof}
Suppose that Conjecture~\ref{conj:obv} holds and $\Sa B$ interprets an infinite field $L$.
By Conjecture~\ref{conj:obv} there is a $K$-definable bijection $L \to K^m$ for some $m$, so for some $k$ there is a $K$-definable surjection $\B^k \to K^m$.
This contradicts Lemma~\ref{lem:map}.
\end{proof}

\noindent
Let $\Sa B(\Q_p)$ be the structure induced on $\B(\Q_p)$ by $\Q_p$.
Conjecture~\ref{conj:obv} implies that $\Sa B(\Q_p)$ does not interpret an infinite field.
By Fact~\ref{fact:delon} $\Sa B(\Q_p)$ is Shelah complete.
This is why we start with a proper elementary extension of $\Q_p$.

\subsection{Dp-rank of $E$}
The examples in \cite{big-nip,field-in-Shelah} are weakly o-minimal, hence dp-rank one.
It is easy to see that the dp-rank of $\Sa B$ is two, but $E$ is dp-rank one.

\begin{prop}
\label{prop:dp-minimal}
Then the dp-rank of $E$ (considered as either a $K$-definable or $\ksh$-definable set) is one.
\end{prop}

\begin{proof}
We apply Fact~\ref{fact:dp-rank}.
Note that $\dprk_{K} E$ agrees with $\dprk_{\ksh} X$.
As $E$ is infinite $\dprk_{K} E \geq 1$.
By Fact~\ref{fact:dp-qp} we have $\dprk_{K} V = 1$. 
The map $V \to E$ given by $a \mapsto B(a,\gamma)$ is $K$-definable and surjective so $\dprk_K E = 1$.
\end{proof}

\noindent
Letting $\Sa E$ be the structure induced on $E$ by $K$ we see that $\Sa E$ has dp-rank one, $\Sh E$ interprets $\Q_p$, and Conjecture~\ref{conj:obv} implies that $\Sa E$ does not interpret an infinite field.

\subsection{The geometric sorts}
We first recall some well-known facts about the ``geometric sorts" introduced in \cite{Haskell2006}.
Let $L$ be a valued field with valuation ring $O$.
An \textbf{$n$-lattice} is a free rank $n$ $O$-submodule of $K^n$.
Let $\Cal S_n(L)$ be the set of $n$-lattices.
It is well known that there is a canonical bijection $\Cal S_n(L) \to \Gl_n(L)/\Gl_n(O)$ so we take $\Cal S_n(L)$ to be an $L$-definable set of imaginaries.
\newline

\noindent
It is shown in \cite{Hrushovski2018} that $K$ eliminates imaginaries down to certain ``geometric sorts".
One of the two kinds of ``geometric sorts" is $\mathcal{S}_n(K)$.
We describe the canonical injection $\B \to \Cal S_2(K)$.
Fix $B \in \B$.
Let $R$ be the $V$-submodule of $K^2$ generated by $\{1\} \times B$.
It is easy to see that $R$ is a $2$-lattice and $R \cap [\{1\} \times K] = B$.
\newline

\noindent
Note that $\valp : K^\times \to \Gamma$ gives an isomorphism $\Cal S_1(K) \to \Gamma$.
Recall that the structure induced on $\Gamma$ by $K$ is interdefinable with $(\Gamma,+,<)$.
We expect that $(\Gamma,+,<)^{\mathrm{Sh}}$ cannot interpret an infinite field.

\begin{prop}
\label{prop:geometric-sorts}
Let $\Sa S_n$ be the structure induced on $\Cal S_n(K)$ by $K$.
If $n \geq 2$ then $\Sh S_n$ interprets $\Q_p$.
Conjecture~\ref{conj:obv} implies that $\Sa S_n$ does not interpret an infinite field.
\end{prop}

\noindent
Hrushovski and Rideau~\cite{metastable} show that if $L$ is an algebraically closed valued field then the induced structure on $\Cal S_n(L)$ does not interpret an infinite field, but their techniques do not directly apply to the $p$-adic case.

\begin{proof}
Fix $n \geq 2$.
As there is a $K$-definable injection $\B \to \Cal S_2(K)$ and a natural injection $\Cal S_2(K) \to \Cal S_n(K)$ the first claim follows by Theorem~\ref{thm:2}.
The second claim follows by the proof of Theorem~\ref{thm:1} as $\Cal S_n(\Q_p) = \Gl_n(\Q_p)/\Gl_n(\Z_p)$ is countable.
\end{proof}

\section{Trace definibility}
\noindent
We discuss trace definibility and in particular prove Theorem~\ref{thm:field-ext}.
The first two claims of Proposition~\ref{prop:trace-basic} are immediate.
The third is an easy corollary to Fact~\ref{fact:shelah}.

\begin{prop}
\label{prop:trace-basic}
Let $\Sa M, \Sa O$ and $\Sa P$ be structures.
\begin{enumerate} 
\item Suppose that $\Sa O$ is trace definable in $\Sa M$ and $\Sa P$ is trace definable in $\Sa O$.
Then $\Sa P$ is trace definable in $\Sa M$.
\item If $\Sa O \prec \Sa M$ then $\Sa O$ is trace definable in $\Sa M$.
\item If $\Sa M$ is $\nip$, $\Sa O \prec \Sa M$, and $\Sa M$ is $|O|^+$-saturated, then $\Sa M$ trace defines $\Sh O$.
\end{enumerate}
\end{prop}

\noindent
We view trace definibility as a weak notion of interpretability.

\begin{prop}
\label{prop:trace-interpret}
If $\Sa O$ is interpretable in $\Sa M$ then $\Sa O$ is trace definable in $\Sa M$.
\end{prop}

\noindent
In the proof below $\pi$ will denote a certain map $E \to O$ and we will also use $\pi$ to denote the map $E^n \to O^n$ given by
$ \pi(a_1,\ldots,a_n) = (\pi(a_1),\ldots,\pi(a_n))$.

\begin{proof}
Suppose $\Sa O$ is interpretable in $\Sa M$.
Let $E \subseteq M^m$ be  $\Sa M$-definable, $\approx$ be an $\Sa M$-definable equivalence relation on $E$, and $\pi : E \to O$ be a surjection such that
\begin{enumerate} 
\item for all $a,b \in E$ we have $a \approx b$ if and only if $\pi(a) = \pi(b)$, and
\item if $X \subseteq O^n$ is $\Sa O$-definable  then $\pi^{-1}(X)$ is $\Sh M$-definable.
\end{enumerate}
Let $\tau : O \to E$ be a section of $\pi$ and $A = \tau(O)$.
So $A$ contains exactly one element from every $\approx$-class.
If $X \subseteq O^n$ is $\Sa O$-definable then $\pi^{-1}(X)$ is $\Sa M$-definable and $X$ is the set of $a \in O^n$ such that $\tau(a) \in \pi^{-1}(X)$.
\end{proof}

\begin{prop}
\label{prop:trace-new}
Let $\lambda$ be an uncountable cardinal.
Suppose that $\Sa M$ is $\lambda$-saturated and $\nip$.
Suppose $\Sh M$ trace defines $\Sa O$ and $|O| < \lambda$.
Then $\Sa M$ trace defines $\Sa O$.
\end{prop}

\begin{proof}
We may suppose that $O \subseteq M^m$ and that for every $\Sa O$-definable $X \subseteq O^n$ there is an $\Sh M$-definable $Y \subseteq M^{mn}$ such that $O^n \cap Y = X$.
Fix an $\Sa O$-definable subset $X$ of $O^n$.
Let $Y$ be an $\Sh M$-definable subset of $M^{mn}$ such that $O^n \cap Y = X$.
Applying Facts~\ref{fact:shelah} and \ref{fact:cs} we obtain an $\Sa M$-definable family $(Z_b : b \in M^k)$ such that for every finite $B \subseteq Y$ we have $B \subseteq Z_b \subseteq Y$.
So for any finite $B \subseteq X$ there is $b \in M^k$ such that $B \subseteq Z_b$ and $Z_b \cap [O^n \setminus X] = \emptyset$.
As $|O^n| < \lambda$ an application of saturation yields $b \in M^k$ such that $X = O^n \cap Z_b$.
\end{proof}

\noindent
Theorem~\ref{thm:field-ext} is a special case of Proposition~\ref{prop:field-ext-str}.

\begin{prop}
\label{prop:field-ext-str}
Let $\lambda$ be an uncountable cardinal.
Suppose that $\Sa M$ is $\lambda$-saturated and $\nip$.
Suppose that $\Sh M$ interprets $\Sa O$.
If $\Sa P \prec \Sa O$ and $|P| < \lambda$ then $\Sa M$ trace defines $\Sa P$.
In particular if $\Sa M$ is $\aleph_1$-saturated and $\Sh M$ interprets an infinite field then $\Sa M$ trace defines an infinite field.
\end{prop}

\noindent
In \cite{field-in-Shelah} we described a $(2^{\aleph_0})^+$-saturated $\nip$ structure $\Sa H$ such that $\Sa H$ does not interpret an infinite group but $\Sh H$ interprets $(\R,+,\times)$.
So $\Sa H$ trace defines $(\R,+,\times)$.
The example presented in \cite{big-nip} also trace defines $(\R,+,\times)$ for the same reason.

\begin{proof}
The second claim follows from the first claim and L\"owenheim-Skolem.
Let $\Sa P$ be an elementary substructure of $\Sa O$ such that $|P| < \lambda$.
By Proposition~\ref{prop:trace-interpret} $\Sh M$ trace defines $\Sa O$.
The first two items of Proposition~\ref{prop:trace-basic} show that $\Sh M$ trace defines $\Sa P$.
So $\Sa M$ trace defines $\Sa P$ by Proposition~\ref{prop:trace-new}.
\end{proof}



\noindent
Finally, we show that if $\Sa M$ is $\nip$ then any structure which is interpretable in the Shelah completion of $\Sa M$ is trace definable in an elementary extension of $\Sa M$.

\begin{prop}
\label{prop:trace-final}
Suppose that $\Sa M$ is $\nip$ and that $\Sh M$ trace defines $\Sa O$.
Suppose that $\Sa N$ is an $|M|^+$-saturated elementary extension of $\Sa M$.
Then $\Sa N$ trace defines $\Sa O$.
In particular if $\Sh M$ interprets $\Sa P$ then $\Sa N$ trace defines $\Sa P$.
\end{prop}

\begin{proof}
By Proposition~\ref{prop:trace-basic}$(3)$ $\Sa N$ trace defines $\Sh M$, so by Proposition~\ref{prop:trace-basic}$(1)$ $\Sa N$ trace defines $\Sa O$.
The second claim follows from the first claim by Proposition~\ref{prop:trace-interpret}.
\end{proof}

\noindent
In the remainder of this section we make a few more observations about trace definibility which are not directly connected to the main topic of this paper.

\subsection{Trace definibility and tameness properties}
A number of positive model-theoretic properties are equivalent to non-trace definibility of a particular countable homogeneous relational structure with a finite language.
We discuss two cases here: stability and $\nip$.
\newline

\noindent
It is easier to trace define relational structures with quantifier elimination.
We leave the verification of Proposition~\ref{prop:qe} to the reader.

\begin{prop}
\label{prop:qe}
Suppose that $L$ is a relational language, $T$ is an $L$-theory which admits quantifier elimination, $\Sa O \models T$, and $\tau : O \to M^m$ is an injection.
Suppose that for every $n$-ary relation symbol $R \in L$ there is a $\Sa M$-definable $Y \subseteq M^{mn}$ such that for all $(a_1,\ldots,a_n) \in O^n$ we have
$$ \Sa O \models R(a_1,\ldots,a_n) \quad \text{if and only if} \quad (\tau(a_1),\ldots,\tau(a_n)) \in X. $$
Then $\Sa M$ trace defines $\Sa O$ via $\tau$.
\end{prop}

\noindent
Proposition~\ref{prop:trace} is easy and left to the reader.

\begin{prop}
\label{prop:trace}
If $\Sa M$ is stable then any structure which is trace definable in $\Sa M$ is stable.
If $\Sa M$ is $\nip$ then any structure which is trace definable in $\Sa M$ is $\nip$.
\end{prop}

\noindent
The \textbf{random graph} is as usual the Fra\"iss\'e limit of the class of finite (symmetric irreflexive) graphs.
Recall that the theory of the random graph is $\mathrm{IP}$ and has quantifier elimination.

\begin{prop}
\label{prop:trace-0}
Let $\Sa M$ be $\aleph_1$-saturated.
Then $\Sa M$ is unstable if and only if $\Sa M$ trace defines $(\Q,<)$ and $\Sa M$ is $\mathrm{IP}$ if and only if $\Sa M$ trace defines the random graph.
\end{prop}

\begin{proof}
The right to left direction of both claims follows by Proposition~\ref{prop:trace}.
Suppose that $\Sa M$ is unstable.
Applying $\aleph_1$-saturation we obtain a sequence $(a_q : q \in \Q)$ of elements of some $M^m$ and a formula $\phi(x,y)$ such that for all $p,q \in \Q$ we have $\Sa M \models \phi(a_p,a_q)$ if and only if $p < q$.
Let $\tau : \Q \to M^m$ be given by declaring $\tau(q) = a_q$ for all $q \in \Q$.
As $(\Q,<)$ admits quantifier elimination an application of Proposition~\ref{prop:qe} shows that $\Sa M$ trace defines $(\Q,<)$ via $\tau$.
\newline

\noindent
Suppose that $\Sa M$ is $\mathrm{IP}$.
Applying \cite[Lemma 2.2]{Laskowski-shelah-karp} and $\aleph_1$-saturation there is an isomorphic copy $(V,E)$ of the random graph such that $V \subseteq M^m$ for some $m$ and a formula $\varphi(x,y)$ such that for all $a,b \in V$ we have $(a,b) \in E$ if and only if $\Sa M \models \varphi(a,b)$.
An application of Proposition~\ref{prop:qe} shows that $\Sa M$ trace defines $(V,E)$ via the identity $V \to M^m$.
\end{proof}

\noindent
We conjecture that any infinite field trace definable in an o-minimal structure is either real closed or algebraically closed and any infinite field trace definable in $K$ is a finite extension of a $p$-adically closed field.
Proposition~\ref{prop:trace-distal} is a tiny step in this direction (both o-minimal structures and $p$-adically closed fields are distal.)

\begin{prop}
\label{prop:trace-distal}
A distal structure cannot trace define an infinite field of positive characteristic.
\end{prop}

\noindent
Chernikov and Starchenko~\cite{CS} show that a distal structure cannot interpret an infinite field.
Proposition~\ref{prop:trace-distal} follows by the same proof.
A structure $\Sa O$ satisfies the strong Erd\"os-Hajnal property if for every $\Sa O$-definable subset $X$ of $O^m \times O^n$ there is a real number $\delta > 0$ such that for any finite $A \subseteq M^m, B \subseteq M^n$ there are $A' \subseteq A$, $B' \subseteq B$ such that $|A'| \geq \delta |A|$, $|B'| \geq \delta B$, and $A' \times B'$ is either contained in or disjoint from $X$.
Proposition~\ref{prop:eh-trace} is clear from the definitions.

\begin{prop}
\label{prop:eh-trace}
If $\Sa M$ has the strong Erd\"os-Hajnal property than any structure trace definable in $\Sa M$ has the strong Erd\"os-Hajnal property.
\end{prop}

\noindent
Chernikov and Starchenko~\cite{CS} show that any distal structure has the strong Erd\"os-Hajnal property.
They also observe that the failure of the strong Erd\"os-Hajnal property for infinite fields of positive characteristic is a direct consequence of well-known facts from incidence combinatorics over finite fields together with the theorem of Kaplan, Scanlon, and Wagner~\cite{Kaplan2011} that an infinite $\nip$ field of positive characteristic contains the algebraic closure of its prime subfield.
Proposition~\ref{prop:trace-distal} follows from these facts and Proposition~\ref{prop:trace}.
\newline

\noindent
Another natural conjecture is that a divisible ordered abelian group cannot trace define $(\R,<,+,t \mapsto \lambda t)$ for any $\lambda \in \R \setminus \Q$.

\section{Modularity?}
\noindent
This paper is motivated by the following question.
$$ \textit{Is there a good notion of modularity for $\nip$ structures?} $$
There is a good notion of modularity for stable structures which we refer to as ``one-basedness".
There is also a good notion of modularity for o-minimal structures, an o-minimal structure is modular if and only if algebraic closure is locally modular.
Fact~\ref{fact:ps} follows from the Peterzil-Starchenko trichotomy~\cite{PS-Tri}.

\begin{fact}
\label{fact:ps}
Suppose $\Sa M$ is an o-minimal expansion of an ordered group.
Then the following are equivalent:
\begin{enumerate}
\item algebraic closure in $\Sa M$ is locally modular,
\item $\Sa M$ is a reduct of an ordered vector space over an ordered division ring,
\item $\Sa M$ does not interpret an infinite field.
\end{enumerate}
\end{fact}

\noindent
Here $(1)$ is an abstract modularity notion, $(2)$ asserts that definable sets are ``affine objects" in a reasonable sense, and $(3)$ asserts the absence of a certain algebraic structure.
Analogues of Fact~\ref{fact:ps} should hold for other well-behaved classes of $\nip$ structures.
Examples such as $\Sa B$ indicate that $(3)$ should be replaced with ``$\Sh M$ does not interpret an infinite field" or ``$\Sa M$ does not trace define an infinite field".
In \cite{big-nip} we gave an example of a weakly o-minimal expansion $\Sa Q$ of $(\Q,+,<)$ such that algebraic closure in $\Sa Q$ agrees with algebraic closure in $(\Q,+,<)$ and if $\Sa Q \prec \Sa N$ is $(2^{\aleph_0})^+$-saturated then $\Sh N$ interprets $(\R,+,\times)$.
So analogues of Fact~\ref{fact:ps} will requires a new notion of modularity which is not defined in terms of algebraic closure.
\newline

\noindent
We record some reasonable conditions for a notion of modularity for $\nip$ structures (but we do not insist that all of these conditions be satisfied, see the remarks below).
\begin{enumerate}[label=(A\arabic*)]
    \item Modularity implies $\nip$,
    \item Modularity is preserved under elementary equivalence.
    \item Infinite fields are not modular.
    \item A structure is modular if and only if its Shelah completion is modular.
    \item Ordered abelian groups and ordered vector spaces are modular.
    \item An o-minimal structure is modular if and only if it is modular in the sense of Peterzil-Starchenko.
    \item Monadically $\nip$ structures are modular (so trees are modular and the expansion of a linear order $(O,<)$ by any collection of monotone functions $O \to O$ is modular).
    \item Reasonable valued abelian groups are modular.
    \item Any structure which is trace definable in a modular structure is modular (so in particular modularity is preserved under interpretations).
\end{enumerate}

\noindent
Note that $($A$4)$ follows from $($A$2)$, $($A$9)$, and Proposition~\ref{prop:trace-basic}.
But it is possible that $($A$9)$ is too strong, as one-basedness is not preserved under reducts, see~\cite{Evans2005}.
However, we do not insist that modularity and one-basedness agree over stable theories, so it may be the case that a reduct of a one-based structure is always modular.
\newline

\noindent
A structure $\Sa M$ is \textbf{monadically $\nip$} if the expansion of $\Sa M$ by \textit{all} subsets of $M$ is $\nip$.
See \cite{Simon-dp} for a proof that trees and expansions of linear orders by monotone functions are monadically $\nip$.
\newline

\noindent
The ``reasonable" in $($A$8)$ is necessary as any structure is interpretable in some valued abelian group, see Schmitt~\cite{Schmitt1984}.
It is not clear what ``reasonable" should mean, at a minimum the valued additive group of a valued field should be modular.
We also expect the expansion of $(\Z,+)$ by \textit{all} $p$-adic valuations to be modular (this structure is $\nip$ by \cite{AldE}).
\newline



\noindent
Finally, we record a question of Hrushovski.
We do not believe that this question has appear in print.
Question~\ref{ques:hrush} is motivated by the theorem of Chernikov and Starchenko~\cite{CS} that a distal structure cannot interpret a field of positive characteristic (see Proposition~\ref{prop:trace-distal} above).

\begin{qst}
\label{ques:hrush}
Let $\F$ be a finite field, $V$ be an $\F$-vector space, and $\Sa V$ be a distal expansion of $V$.
Must $\Sa V$ be (in some sense) modular?
\end{qst}

\noindent
A more precise question: Can $\Sa V$ trace define an infinite field? 
\newpage

\bibliographystyle{abbrv}
\bibliography{NIP}
\end{document}